\documentclass[a4paper,10pt,twoside]{article}

\usepackage{amssymb}
\usepackage{amsmath}
\usepackage[dvips]{epsfig}

\newcommand{\bv}[1]{{\mbox{\boldmath $ #1$}}}

\def\a{\bv{a}}

\def\v{\bv{v}}

\def\cb{\bv{c}}
\def\db{\bv{d}}

\def\cbar{\bv{\bar{c}}}

\def\diag{{\rm diag}}
\def\sign{{\rm sgn}}
\def\Real{\mathbb{R}}  
\newcommand{\lbeq}[1]{{\label{OR:eq:#1}}}
\newcommand{\eq}[1]{{(\ref{OR:eq:#1})}}
\newcommand{\eqtwo}[2]{{(\ref{OR:eq:#1}, \ref{OR:eq:#2})}}
\newcommand{\be}[1]{\begin{equation} \lbeq{#1}}
\newcommand{\ee}{\end{equation}}
\newcommand{\Vector}[1]{{\left(\begin{matrix} #1 \end{matrix}\right)}}

\newtheorem{remark}{Remark}
\newtheorem{proposition}{Proposition}
\newtheorem{theorem}{Theorem}
\newtheorem{definition}{Definition}
\newtheorem{lemma}{Lemma}
\newtheorem{algorithm}{Algorithm}
\newtheorem{corollary}{Corollary}
\newenvironment{proof}{\textit{Proof:}\ }{$~\Box$}

\def\Deg{{\rm Deg}}
\def\Dmin{D_{\rm min}}
\def\Dmax{D_{\rm max}}
\def\Span{{\rm span}}
\def\Rank{{\rm rank}\;}

\title{Existence, uniqueness and a constructive solution algorithm for a class of finite Markov moment problems}

\author{
Laurent Gosse\footnote{IAC--CNR ``Mauro Picone" (sezione di Bari), Via Amendola 122/D, 70126 Bari (Italy)
{\tt l.gosse@ba.iac.cnr.it}}
 \and
Olof Runborg\footnote{CSC, KTH, 10044 Stockholm (Sweden) {\tt olofr@nada.kth.se}}
}
\date{\today}
\begin{document}

\maketitle

\begin{abstract}
We consider a class of finite Markov moment problems with arbitrary
number of positive and negative branches. We show criteria for the existence
and uniqueness of solutions, and we characterize in detail the
non-unique solution families. Moreover, we present a constructive algorithm
to solve the moment problems numerically and prove that the algorithm
computes the right solution.
\end{abstract}

\section{Introduction}

We aim at inverting a moment system often associated with the prestigious name of Markov.
The original form of the problem is the following.
Given a {\em finite set of moments} $m_k$ for $k=1,\ldots,K$, find a 
bounded measurable
{\em density} function $f$ satisfying 
\be{integrel}
   m_k = \int_{\Real} x^{k-1}f(x)dx, \qquad 0\leq f\leq 1,
\qquad k=1,\ldots, K.
\ee
Condition for the existence of solutions $f(x)$ to this problem
is classical \cite{akhiezer,ak-krein}. In general solutions are not
unique, unless more conditions are given, e.g. based on entropy minimization
\cite{poly,brecor} or $L^{\infty}$-minimization \cite{norris,lewis}.
A typical result is that the unique solution for even $K$ is 
piecewise constant, taking values in $\{0,1\}$. More precisely, 
if $K=2n$ then $f$ is of the form
\be{fMarkov}
   f(x) = \sum_{j=1}^n \chi_{[y_i,x_i]}(x)
\ee
where $\chi_I(x)$ is the characteristic function for the interval $I$
and 
\be{interlace}
   y_1<x_1<y_2<x_2<\cdots<y_n<x_n.
\ee
See Theorem \ref{markov} below in Section \ref{sec:exist}
and
consult e.g. \cite{CF,diaco,krein,simone,talenti} for general background 
on moment problems. 

A reduced form of the finite moment problem 
is to search for solutions to \eq{integrel} which are precisely of
the form \eqtwo{fMarkov}{interlace}.
One then obtains an algebraic problem for the branch values,
\be{mom-olof}
  m_k = \frac1k\sum_{j=1}^{n} x_{j}^k-y_{j}^k,
\qquad k=1,\ldots, K=2n.
\ee
Finding $\{x_j\}$ and $\{y_j\}$ from $\{m_k\}$ is
an ill-conditioned problem when the branch values of the solution
come close to each 
other; the Jacobian of the problem 
is a Vandermonde matrix and iterative numerical 
resolution routines require extremely good starting guesses when
the matrix degenerates. 
For less than four moments a direct method based on
solving polynomial equations was presented in \cite{olof1}. Routines
based on the Simplex algorithm have been proposed in \cite{norris}.
Another algorithm was
presented by Koborov, Sklyar and Fardigola
in \cite{kobo1,sklyar1} in the slightly modified setting
where $f$ takes values in $\{-1,1\}$ instead of $\{0,1\}$.
It consists of
solving a sequence of high degree polynomial equations,
constructed through a rather intricate process with unclear stability
properties. In \cite{GO} we showed that this algorithm 
can be drastically simplified and adapted to \eq{mom-olof}.
Later, in \cite{GO2}, we also gave a direct proof that the simplified algorithm indeed computes the correct solution, relying on the classical 
Newton's identities and Toeplitz matrix theory.

The moment problem has many applications in for instance
probability and statistics \cite{GL,persil}, but also in areas like
wave modulation \cite{ieee,sez}
and ``shape from moments" inverse problems \cite{GM}.
Our own motivation comes from a quite different
field, namely multiphase geometrical optics \cite{poly,brecor,G,GJL,GO,olof1}. 
In this
application one needs to solve a system of nonlinear hyperbolic conservation
laws. To evaluate the flux function
in the partial differential equations (PDEs) a system like \eq{mom-olof} 
must be solved. In a finite difference method this means that
the system must be inverted once for every point in the computational
grid, repeatedly in every timestep.
It is thus important that the inversion can be done fast and accurately;
this difficulty has been a bottleneck in computations.
In \cite{GO} we used the simplified algorithm 
mentioned above for numerical implementation inside a 
shock-capturing finite difference solver.
It is our aim here to develop better algorithms and understanding to
open the way for the processing of intricate wave-fields with large $K$,
and thus complement the seminal paper \cite{brecor} where
the multiphase geometrical optics PDEs were first proposed.

In this paper we are concerned with a generalization of
\eq{mom-olof}. In the geometrical optics application,
the number of moments $K$ is typically not even and
one can have a variable number of positive ($x_k$) and negative
($y_k$) branches. We thus consider the following problem
\be{mom-olof2}
  m_k = \sum_{j=1}^{n_x} x_{j}^k- \sum_{j=1}^{n_y} y_{j}^k,
\qquad k=1,\ldots, K,
\ee
where $n_x+n_y=K$ but where $n_x$ and $n_y$ are not necessarily equal.
We study
existence and uniqueness of solutions to this problem
(Theorem \ref{existence}). 
In particular we are interested in how and when uniqueness is
lost. For these cases we characterize the family of
solutions that exists. The reason is to understand what
happens numerically close to degenerate solutions, which
is an important feature in the application we have in mind:
In the exact solution to the multiphase geometrical optics PDEs the
moment problem is typically degenerate for large domains;
the numerical approximation is almost degenerate.

We also give constructive algorithms to solve
 \eq{mom-olof2} and prove that they generate the right
solution (Theorem \ref{solvalgo}).
In a future paper we will study the numerical stability
of these algorithms. 
Experimentally we note, for instance,
that to compute the next moment,
Algorithm 3 is much more stable than Algorithm 1.
The difficulty lies in understanding
perturbations around degenerate solutions, which is where
the algorithms are most unstable.  
For this the insights of this paper will be
of importance. 
\begin{remark}
  The problem \eq{mom-olof2} can be cast in the form of
\eq{integrel} if one demands that the density function
$f(x)$ is of the form
\be{fWe}
    f(x) = \sum_{j=1}^{n_x}\sign(x_j)\left[H(x)-H(x-|x_j|)\right]
- \sum_{j=1}^{n_y}\sign(y_j)\left[H(x)-H(x-|y_j|)\right],
\ee
and we rescale the moments $m_k\to km_k$. For 
the case 
$n_x=n_y=n$ and $K=2n$ 
with interlaced branch values \eq{interlace}
this reduces to \eq{fMarkov}.
\end{remark}

This paper is organized as follows.
In Section \ref{sec:algo} we present the algorithms
for solving \eq{mom-olof2}. Notation and various ways
of describing a solution is subsequently
introduced in Section \ref{sec:prelim}.
 Next we derive conditions
for existence and uniqueness of solutions in Section \ref{sec:exist}
and also discuss various properties of the solution, in particular
when it is not unique. A theorem proving the correctness of
the algorithms is proved in Section \ref{sec:T1proof}.
Finally, in Section \ref{sec:Markov},
we give additional properties of the elements of our algorithms, and
use these to relate our results back to the classical
Markov theory.

\section{Algorithms}\label{sec:algo}

In this section we detail the algorithms that we propose for solving
\eq{mom-olof2}. 
The solution that we obtain is what we call the {\em minimal
degree solution}, meaning that when the solution is not unique
as many branch values as possible are zero. See Section
\ref{sec:exist} for a precise definition.
The algorithms goes as
follows; they  may fail in case there is no solution to \eq{mom-olof2}.
\begin{algorithm}[Computing $\{x_j\}$ and $\{y_j\}$]\label{algo1}
\mbox{}
\begin{enumerate}
\item Construct the sequence $\{a_k\}$ as follows.
Set $a_0=1$ and $a_k=0$ for $k<0$. For $1\leq k \leq K$,
let the elements be given as the solution to
\be{adef}
   \Vector{1 & & & \\
-m_1 & 2 &\\
\vdots &\ddots &\ddots & \\
-m_{K-1} & \hdots & -m_1 & K
}
\Vector{a_1 \\ a_2 \\ \vdots \\ a_K} =
\Vector{m_1 \\ m_2 \\ \vdots \\ m_K}.
\ee

\item Construct the matrix $A_1\in\Real^{n_x\times n_x}$ as
$$ 
   A_1=\Vector{a_{{n_y}} & a_{n_y-1} & \hdots & a_{n_y-n_x+1} \\
 a_{n_y+1} & a_{n_y} & \hdots & a_{n_y-n_x+2}\\
            \vdots & \vdots & \ddots & \vdots \\
a_{n_y+n_x-1}
 & a_{n_y+n_x-2} & \hdots & a_{n_y}
}.
$$
Compute the rank of $A_1$. Let $\tilde{n}_x=\Rank A_1$
and $\tilde{n}_y = n_y - n_x +\tilde{n}_x$.

\item Construct the matrices $\tilde{A}_0,\tilde{A}_1\in\Real^{\tilde{n}_x\times \tilde{n}_x}$ as
$$ 
   \tilde{A}_0=\Vector{a_{{\tilde{n}_y+1}} & a_{\tilde{n}_y} & \hdots & a_{\tilde{n}_y-\tilde{n}_x+2} \\
 a_{\tilde{n}_y+2} & a_{\tilde{n}_y+1} & \hdots & a_{\tilde{n}_y-\tilde{n}_x+3}\\
            \vdots & \vdots & \ddots & \vdots \\
a_{\tilde{n}_y+\tilde{n}_x}
 & a_{\tilde{n}_y+\tilde{n}_x-1} & \hdots & a_{\tilde{n}_y+1}
},
$$
$$
   \tilde{A}_1=\Vector{a_{{\tilde{n}_y}} & a_{\tilde{n}_y-1} & \hdots & a_{\tilde{n}_y-\tilde{n}_x+1} \\
 a_{\tilde{n}_y+1} & a_{\tilde{n}_y} & \hdots & a_{\tilde{n}_y-\tilde{n}_x+1+1}\\
            \vdots & \vdots & \ddots & \vdots \\
a_{\tilde{n}_y+\tilde{n}_x-1}
 & a_{\tilde{n}_y+\tilde{n}_x-2} & \hdots & a_{\tilde{n}_y}
}.
$$
\item Solve the generalized eigenvalue problem 
\be{geneig}
\tilde{A}_0\v = x \tilde{A}_1\v,
\ee
to get the $\{x_j\}$ values of the minimal degree solution
to \eq{mom-olof2}.
\item To compute the $\{y_j\}$ values, the same process is used with $m_k$ replaced by $-m_k$
and the roles of $n_x$ and $n_y$ interchanged.
\end{enumerate}
\end{algorithm}
An alternative to Algorithm \ref{algo1} is as follows:
\begin{algorithm}[Computing $\{x_j\}$ and $\{y_j\}$]\label{algo2}
\mbox{}
\begin{enumerate}
 \item Construct the matrices $\tilde{A}_0$ and $\tilde{A}_1$ as in steps 1-3 in Algorithm \ref{algo1}.
\item Denote the first column vector in $\tilde{A}_0$ by $\tilde{\a}_0$ by 
and solve
\be{cprimdef}
   \tilde{A}_1\cb'=-\tilde{\a}_0, \qquad \cb'=(c_1,c_2,\ldots, c_{\tilde{n}_x})^T.
\ee
\item Construct the polynomial 
$$
   P(z) = c_{\tilde{n}_x} + c_{\tilde{n}_x-1}z + \cdots + 
c_{1}z^{\tilde{n}_x-1} + z^{\tilde{n}_x}.
$$
The roots of $P(z)$ are
the $\{x_j\}$ values of
the minimal degree solution to \eq{mom-olof2}
(possibly together with some zeros).
\item To compute the $\{y_j\}$ values, the same process is used with $m_k$ replaced by $-m_k$
and the roles of $n_x$ and $n_y$ interchanged.
\end{enumerate}
\end{algorithm}
\begin{remark}
We note that the values of $a_k$ in the definition \eq{adef} are independent
of $K$, since the system matrix is triangular. We therefore consider
the sequence without reference to $K$ in any other respect than the
fact that we are only able to compute elements with $k\leq K$ 
when we are given $K$ moments. The largest index of the $a_k$-sequence
appearing
in the matrix $A_1$ is $n_y+n_x-1<K$. In the matrices $\tilde{A}_0,\tilde{A}_1$ 
it is $\tilde{n}_y+\tilde{n}_x=n_y - n_x +2\tilde{n}_x\leq n_y+n_x=K$.
Hence all three matrices can be constructed from the first $K$ moments.
Some properties of the $A_1$ matrix are detailed in 
Section \ref{sec:Markov}.
\end{remark}

Sometimes one is not interested in finding the individual
$\{x_j\}$ and $\{y_j\}$ branch values but just wants the higher moments,
defined as 
\be{highmom}
  m_k = 
\sum_{j=1}^{n_x} x_{j}^{k}- \sum_{j=1}^{n_y} y_{j}^{k},
\ee
but now for $k>K$, {\em given} a solution $\{x_j\}\cup\{y_j\}$ to
\eq{mom-olof2}. (That this is well-defined is shown later
in Theorem \ref{existence}.) For this case there is another algorithm,
which has empirically proven to be
more stable than first computing $\{x_j\}$ and
$\{y_j\}$ from Algorithm \ref{algo1} or \ref{algo2},
and then entering the values into \eq{highmom}. We stress that this is
precisely what is needed in order to compute $K$-multivalued solutions
of the inviscid Burger's equation in geometrical optics,
following the ideas of \cite{brecor}.
\begin{algorithm}[Computing $m_{K+1}$]\label{algo3}
\mbox{}
\begin{enumerate}
\item Construct the $A_1$ matrix as in steps 1-2 of Algorithm \ref{algo1}.
\item Let
$$ 
\a_0 = (a_{{{n}_y+1}} , a_{{n}_y+2} , \ldots , a_{{n}_y+{n}_x})^T
\in\Real^{n_x},
$$
and let $\cbar=(c_1,c_2,\ldots, c_{{n}_x})^T$ be one solution to
\be{cbardef}
   {A}_1\cbar=-{\a}_0.
\ee

\item The next moment is given by
$$
  m_{K+1}  
 = -(K+1)\sum_{j=1}^{n_x}c_ja_{K+1-j} - 
\sum_{j=1}^{K}m_ja_{K+1-j}.
$$
\end{enumerate}
\end{algorithm}

We recall that Algorithm \ref{algo1} has been shown to be numerically
efficient in the paper \cite{GO}.
The justification of these algorithms is given in
Section \ref{sec:T1proof} where we
show the following theorem:
\begin{theorem}\label{solvalgo}
If a solution to \eq{mom-olof2} exists then:
\begin{itemize}
\item[\rm (i)]
In Algorithm \ref{algo1}, the matrix $\tilde{A}_1$ is
non-singular. The generalized eigenvalue problem in
\eq{geneig} is well-defined and 
the generalized eigenvalues (counting algebraic multiplicity)
are the $\{x_j\}$-values of the minimal degree solution to \eq{mom-olof2}
plus $\tilde{n}_x-\Dmin$ zeros. (See \eq{Dmin} for the definition of $\Dmin$.)
\item[\rm (ii)]
In Algorithm \ref{algo2}, $\cb'$ is well defined,
\be{charpol}
   P(z) = \det(zI-\tilde{A}_1^{-1}\tilde{A}_0)
\ee
and the roots of $P(z)$ are
the $\{x_j\}$-values of the minimal degree solution to \eq{mom-olof2}
plus $\tilde{n}_x-\Dmin$ zeros.
\item[\rm (iii)]
In Algorithm \ref{algo3}, the computed moment satisfies
$$
  m_{K+1} = \sum_{j=1}^{n_x} x_{j}^{K+1}- \sum_{j=1}^{n_y} y_{j}^{K+1},
$$
for all solutions $\{x_j\}\cup \{y_j\}$ to \eq{mom-olof2}.
\end{itemize}
\end{theorem}
We postpone the proof of Theorem \ref{solvalgo} to 
Section \ref{sec:T1proof}. We just note here that the last point in
Algorithms 1 and 2 can easily be explained by the symmetry
of the problem. Indeed, the negative of \eq{mom-olof2}
$$
  -m_k = \sum_{j=1}^{n_y} y_{j}^k-\sum_{j=1}^{n_x} x_{j}^k,
\qquad k=1,\ldots, K,
$$
is of the same form as \eq{mom-olof2} itself, with the
roles of $n_x$, $\{x_j\}$ and $n_y$, $\{y_j\}$ interchanged.


\section{Preliminaries}
\label{sec:prelim}


We will use three different ways of describing the
solution to \eq{mom-olof2}. First we have a set of numbers
$\{x_j\}_{j=1}^{n_x}$ and $\{y_j\}_{j=1}^{n_y}$, solving \eq{mom-olof2}. 
We call those numbers branch values.
Second, we have a pair of polynomials $(p,q)$ of degrees at most $n_x$ and $n_y$ 
respectively in the $z$ variable. Third, we
have a pair of coefficient vectors $\cb=(c_0,\ldots,c_{n_x})^T\in\Real^{n_x+1}$
and $\db=(d_0,\ldots,d_{n_x})^T\in\Real^{n_y+1}$. These three representations
are related as
\be{Pc}
   p(z)=(1-x_1z)\cdots(1-x_{n_x}z)=c_{0} +c_{1}z + \cdots + c_{n_x-1} z^{n_x-1}+c_{n_x} z^{n_x},
\ee
and
\be{Pd}
   q(z)=(1-y_1z)\cdots(1-y_{n_y}z)=d_{0} +d_{1}z + \cdots + d_{n_y-1} z^{n_y-1}+d_{n_y} z^{n_y}.
\ee
It is clear that there is a one-to-one correspondence
between these ways of describing the solution, if
we disregard the ambiguity in the ordering of the
numbers $\{x_j\}$ and $\{y_j\}$. 
Generally, we will use the notation $\Deg(p)$ to denote
the degree of a polynomial $p$, and, for a given
coefficient vector $\cb$, we systematically write
$P_{c}$ to denote the corresponding 
polynomial \eq{Pc}.

\begin{definition}
We call the pair of polynomials $(p,q)$ \underline{a (polynomial) solution}
to \eq{mom-olof2} if 
\begin{enumerate}
\item  The degrees of $p$ and $q$ are at
most $n_x$ and $n_y$,
\be{polsol1}
\Deg(p)\leq n_x, \qquad \Deg(q)\leq n_y,
\ee
\item They are normalized to one at the origin,
\be{polsol2}
p(0)=q(0)=1,
\ee
\item Their roots $\{\tilde{x}_j\}$ and $\{\tilde{y}_j\}$ satisfy
\be{polsol3}
  m_k = \sum_{j=1}^{\Deg(p)} \tilde{x}_{j}^{-k}- \sum_{j=1}^{\Deg(q)} \tilde{y}_{j}^{-k},
\qquad k=1,\ldots, K.
\ee
We note that the roots cannot be zero because of \eq{polsol2}.
\end{enumerate}
\end{definition}
Next:
\begin{definition}
A pair of vectors
$$
\cb=(c_0,\ldots,c_{n_x})^T\in\Real^{n_x+1} 
\mbox{ and } \db=(d_0,\ldots,d_{n_y})^T\in\Real^{n_y+1}
$$
is said to be \underline{a (coefficient) solution} to \eq{mom-olof2} 
if the corresponding pair $(P_{c},P_{d})$ \eq{Pc}--\eq{Pd} realizes a polynomial
solution to \eq{mom-olof2}.
\end{definition}

The number of branch
values are always $n_x$ and $n_y$ respectively. Some
of them may be zero, and they do not need to be distinct. The
number of non-zero branch values are $\Deg(p)$ and $\Deg(q)$
respectively. The degree of a solution can then also be defined.
\begin{definition}
The \underline{degree} of a solution to \eq{mom-olof2}
is the number of non-zero $x_j$-values. This number is equivalent
to $\Deg(p)$.
\end{definition}

Given any polynomial pair satisfying \eq{polsol2}, 
we say that it generates the moment sequence $\{m_k\}$
if $m_k$ is given by \eq{polsol3} for all $k$. 
In turn, each sequence of moments $\{m_k\}$ generates the corresponding
$\{a_k\}$ sequence through \eq{adef}.
We define the big matrix
$$ 
   A=\Vector{a_{{n_y+1}} & a_{n_y} & \hdots & a_{n_y-n_x+1} \\
 a_{n_y+2} & a_{n_y+1} & \hdots &a_{n_y-n_x+2} \\
            \vdots & \vdots & \ddots & \vdots \\
a_{n_y+n_x} & a_{n_y+n_x-1} & \hdots & 
a_{n_y}
}\in\Real^{n_x\times (n_x+1)}.
$$
We let the columns of $A$ be denoted $\a_0,\ldots,\a_{n_x}$ and
we note that
\be{Amdef}
   A = \Vector{ | & & | \\ \a_0 & \cdots & \a_{n_x} \\ | & & | }
 = \Vector{ & & | \\  &   A_0 & \a_{n_x}\\  &  & |}
 = \Vector{ | & &  \\ \a_0 &   A_1 & \\ | &  & },
\ee
Hence, $A_0$ and $A_1$ constitutes the first and last $n_x$ columns of
$A$ respectively. 
When $\a_0\in {\rm range}\; A_1$ and $\a_0\neq 0$, let
\be{Dmin}
   \Dmin = {\rm argmin}_{j> 0}\; \a_0 \in \Span\{\a_1,\ldots,\a_j\},
\ee
and set $\Dmin=0$ if $\a_0=0$. Moreover, define
\be{Dmax}
 \Dmax = \Dmin + n_x - \Rank A_1.
\ee
%


\section{Existence and uniqueness of solutions}
\label{sec:exist}

In this section we prove results on the existence
and uniqueness of solutions to \eq{mom-olof2}.
We aim at establishing the following theorem:
\begin{theorem}\label{existence}
\mbox{}
\begin{itemize}
\item[\rm (i)]
There exists a solution to \eq{mom-olof2} if and only if
\be{rankcond}
   \a_0 \in {\rm range}(A_1).
\ee
\item[\rm (ii)]
If $d$ is the degree of a solution to \eq{mom-olof2}, then
$\Dmin\leq d \leq \Dmax$.

\item[\rm (iii)]
When \eq{rankcond} holds, there is a unique solution $(p^*,q^*)$
of minimal degree $\Dmin$. For this solution, 
$x_j\neq y_i$ for all indices $i,j$ representing non-zero branch values.
Moreover, $\Deg(q^*)\leq n_y-n_x+\Rank A_1$ with equality if $\Dmin<\Rank A_1$.
\item[\rm (iv)]
When \eq{rankcond} holds, a polynomial pair $(p,q)$ is a solution
if and only if $p=p^*r$ and $q=q^*r$  where $r(z)$ is a polynomial satisfying 
$r(0)=1$ and $\Deg(r)\leq \Dmax-\Dmin$.
\item[\rm (v)]
The minimal degree solution is the only solution to \eq{mom-olof2}
if and only if the matrix $A_1$ is non-singular. 
\item[\rm (vi)] Let $\{x_j\}$ and $\{y_j\}$ be a solution to \eq{mom-olof2}.
Then the higher moments defined in \eq{highmom} are well-defined.
\end{itemize}
\end{theorem}

Let us proceed with several remarks:

\begin{remark}
In particular it follows from (i) that there exists a solution
as soon as the matrix $A_1$ is non-singular.
\end{remark}
\begin{remark}
Since \eq{mom-olof2} is a system of polynomial equations of
degree $K$, one could expect there to be a finite number of solutions, typically
$K$ solutions. However, because of the special structure of the
equations there is either one unique solution (when $A_1$ is
non-singular) or inifintely many solutions (when $A_1$ is singular).
\end{remark}
\begin{remark}
The form $(p^*r,q^*r)$ of solutions can also be stated
as follows: All solutions have a core set of values
$\{x_j\}$, $j=1,\ldots,\Deg(p^*)=\Dmin$ and
$\{y_i\}$, $i=1,\ldots,\Deg(q^*)$corresponding to non-zero branch values of
the minimal degree solution, 
where $x_j\neq y_i$ for all those $i,j$.
One can then add an optional set of non-zero branch values
$\{{x}_{\Dmin+j}\}$, and
$\{{y}_{\Deg(q^*)+j}\}$, for $j=1,\ldots,\Dmax-\Dmin$ such
that ${x}_{\Dmin+j}={y}_{\Deg(q^*)+j}$. 
\end{remark}

To prove this theorem we first establish some utility
results in the next subsection. We then derive different
ways of characterizing the solution in Section \ref{charsol}, 
which are subsequently
used to prove Theorem \ref{existence} in Section \ref{T2proof}.

\subsection{Utility results}

We start with a useful lemma on Taylor coefficients for a product of functions:
\begin{lemma}\label{Taylor}
Suppose $f$, $g$ and $h$ are analytic functions in a neighborhood of
zero satsifying $f(z)=g(z)h(z)$. Let $f$ have the Taylor expansion
$$
  f(z) = \sum_{k=0}^\infty f_k z^k,
$$
and let $\{g_k\}$, and $\{h_k\}$ be the corresponding coefficients 
for $g(x)$ and $h(x)$ respectively. Then
\be{Taylorrel}
   f_k = \sum_{j=0}^kg_jh_{k-j}.
\ee
\end{lemma}
\begin{proof}
  Since the functions are analytic the coefficients are given as
$$
   f_k = \frac{1}{k!}\frac{d^k}{dz^k}f(z)\Bigr|_{z=0}
= \frac{1}{k!}\frac{d^k}{dz^k}g(z)h(z)\Bigr|_{z=0}
= \frac{1}{k!}\sum_{j=0}^kc_{jk}g^{(j)}(0)h^{(k-j)}(0),
$$
where $c_{jk}=k!/j!(k-j)!$ are the binomial coefficients.
But $g^{(j)}(0)=j!g_j$ and $ h^{(k-j)}(0)=(k-j)!h_{k-j}$ and therefore
\eq{Taylorrel} follows.
\end{proof}
\begin{remark}
 The sum \eq{Taylorrel} is in fact precisely an elementwise description of
multiplication of a lower triangular $k\times k$ Toeplitz
matrix by a vector. In the notation of \cite{GO2},
it would read ${\bv f} = {\mathcal T}({\bv g}){\bv h}$.
\end{remark}

As was already known by Markov, the exponential transform
of the moment sequence plays an important role in the
analysis of these problems, see {\it e.g.} \cite{akhiezer,ak-krein}. 
We show here that $\{a_k\}$ is a version of
the exponential transform of $\{m_k\}$. 

\begin{lemma}\label{aexponent}
Suppose $\{m_k\}$ is generated by the polynomials $p(z)$ and $q(z)$
and $\{a_k\}$ is generated by $\{m_k\}$.
Let $m(z)$ be defined as
\be{mdef}
  m(z) = m_1z + \frac12 m_2 z^2 + \frac13 m_3z^3 + \cdots.
\ee
Then if \eq{polsol2} holds,
\be{achar1}
   e^{m(z)} = \frac{q(z)}{p(z)} = a_0 + a_1 z + a_2 z^2 +\cdots,
\ee
written as its Taylor expansion around $z=0$.
\end{lemma}
\begin{proof}
  Let us first show that $m(z)$ is a well-defined
analytic function at zero. We have
$$\begin{array}{rcl}
  m(z) & = &\sum_{k=0}^\infty \frac{m_kz^k}{k} \\
        & = &\sum_{k=0}^\infty\sum_{j=1}^{n_x} \frac{x_j^kz^k}{k}
- \sum_{k=0}^\infty\sum_{j=1}^{n_y} \frac{y_j^kz^k}{k} \\
        & = &-\sum_{j=1}^{n_x} \log(1-x_jz)
+ \sum_{j=1}^{n_y} \log(1-y_jz).
\end{array}$$
The last step is allowed when $|z|< 1/\max_{ij}(|x_j|,|y_i|)$,
which is true for small enough $z$ since $p(0)\neq 0$.
This also shows that the function is analytic at zero.
Moreover,
$$
  e^{m(z)} = \frac{\prod_{j=1}^{n_y} (1-y_jz)}
{\prod_{j=1}^{n_x} (1-x_jz)}
= \frac{q(z)}{p(z)}.
$$
Finally, setting $a(z):=\exp(m(z))$ and differentiating gives
$$
   za'(z) = zm'(z)a(z),
$$
where all three functions are analytic at zero.
Let $a(z)$ have the Taylor coefficients $\{\tilde{a}_k\}$.
Then $za'(z) = \tilde{a}_1z + 2\tilde{a}_2z^2 + 3\tilde{a}_3 z^3\cdots$
and clearly $zm'(z) = m_1z + m_2z^2 +\cdots$. By Lemma \ref{Taylor},
for $k\geq 1$,
$$
    k\tilde{a}_k = \sum_{j=1}^k m_j\tilde{a}_{k-j}.
$$
Since $\tilde{a}_0 = q(0)/p(0)=1$, we see that $a_k$ and $\tilde{a}_k$ satisfy
the same non-singular
linear system of equations \eq{adef}, and therefore $a_k=\tilde{a}_k$,
showing \eq{achar1}.
\end{proof}

We now have the following basic characterization of a solution.
\begin{lemma}\label{pqexist}
Suppose ${p}(z)$ and $q(z)$ are two polynomials satisfying \eqtwo{polsol1}{polsol2}.
They form a polynomial solution to \eq{mom-olof2} if and only if
their quotient has the Taylor expansion around $z=0$
\be{pqtaylor}
  \frac{{q}(z)}{{p}(z)}=a_0 +a_1z +\cdots +a_Kz^K + O\left(z^{K+1}\right),
\ee
where $\{a_k\}$ is generated by $\{m_k\}$.
Moreover, if $(p,q)$ is a solution then
$(\bar{p},\bar{q})$ is also a solution if and only if the pair
satisfies \eqtwo{polsol1}{polsol2}
and $\bar{p}/\bar{q}=p/q$ where these fractions are defined.
\end{lemma}
\begin{proof}
Let $\{\tilde{m}_k\}$ be generated by ${p}$ and ${q}$ and suppose \eq{pqtaylor} holds. 
Then, as in the of proof of Lemma \ref{aexponent} for $1\leq k \leq K$
$$
    ka_k = \sum_{j=1}^k \tilde{m}_ja_{k-j}.
$$
Since $\{m_k\}$ satisfy the linear system \eq{adef}, we have after subtraction,
$$
  m_n-\tilde{m}_n = -\sum_{k=1}^{n-1}  (m_k-\tilde{m}_k)a_{n-k},\qquad
    m_1=\tilde{m}_1,
$$
for $n=2,\ldots, K$. By induction $\tilde{m}_k=m_k$ for 
$1\leq k\leq K$, showing that $({p},{q})$
solves \eq{mom-olof2}. On the other hand, if $({p},{q})$
is a solution, then \eq{pqtaylor} must hold by \eq{achar1}
in Lemma \ref{aexponent}.

For the last statement, the ``if'' part is obvious
since both pairs then satisfy \eq{pqtaylor}. To show 
the ``only if'' part, suppose both $(p,q)$ and $(\bar{p},\bar{q})$
are solutions. By definition they satisfy \eqtwo{polsol1}{polsol2},
and by \eq{pqtaylor},
$$
  \frac{\bar{q}(z)}{\bar{p}(z)}-\frac{{q}(z)}{{p}(z)} =
  \frac{\bar{q}(z){p}(z)-\bar{p}(z){q}(z)}{\bar{p}(z){p}(z)}
 = O(z^{K+1}).
$$
Since $\bar{p}(0){p}(0)=1$ we must have that
$(\bar{q}(z){p}(z)-\bar{p}(z){q}(z))/z^{K+1}$ is bounded
as $z\to 0$. But since the degree of $\bar{q}{p}-\bar{p}q$
is at most $K=n_x+n_y$ this is only possible if it is identically zero.
Hence $\bar{q}(z){p}(z)=\bar{p}(z){q}(z)$ which concludes the proof.
\end{proof}

\subsection{Characterization of the solution}
\label{charsol}

In this section we show three Propositions that characterize
solutions to \eq{mom-olof2} in terms of polynomials,
coefficient vectors and the column vectors of the $A$-matrix
in \eq{Amdef}.
We start by  expressing the
uniqueness properties of the solution in terms of
its polyomial representation.
\begin{proposition}\label{pqunique}
Suppose the pairs $(p,q)$ and $(\bar{p},\bar{q})$ are both
polynomial solutions to \eq{mom-olof2}.
Then, 
\begin{enumerate}
\item[\rm (i)] $\Deg(p)-\Deg(q) = \Deg(\bar{p})-\Deg(\bar{q})$.
\item[\rm (ii)]
If $\Deg(\bar{p})\leq\Deg(p)$,
and if there is no polynomial $r(z)$ such that $p=\bar{p}r$,
then there is 
another solution $(\tilde{p},\tilde{q})$ with $\Deg(\tilde{p})<\Deg(p)$.
In particular, if $\Deg(p)=\Deg(\bar{p})$ but $p\neq\bar{p}$, there is 
such a lower degree solution.
\item [\rm (iii)]If $\Deg(\bar{p})\leq\Deg(p)$, any polynomial
pair $(\bar{p}r,\bar{q}r)$ is a solution if $r(z)$ is
a polynomial satisfying
$r(0)=1$ and $\Deg(r)\leq \Deg(p)-\Deg(\bar{p})$.
In particular, if $\Deg(\bar{p})\leq m \leq \Deg(p)$ 
there is a solution $(\tilde{p},\tilde{q})$
with $\Deg(\tilde{p})=m$.
\end{enumerate}
\end{proposition}
\begin{proof}
\begin{enumerate}
\item[\rm (i)] The statement follows directly from Lemma \ref{pqexist},
since $\bar{q}{p}=\bar{p}q$
implies that
$$
  \Deg(\bar{q})+\Deg(p) = 
  \Deg(\bar{p})+\Deg(q).
$$

\item[\rm (ii)] We let
$$
   p(z) = r_p(z)\bar{p}(z) + s_p(z), \qquad
   q(z) = r_q(z)\bar{q}(z) + s_q(z),
$$
be the unique polynomial decomposition of $(p,q)$ such
that $r_p,r_q,s_p,s_q$ are polynomials,
$\Deg(s_p)<\Deg(\bar{p})$ and
$\Deg(s_q)<\Deg(\bar{q})$. Since $\bar{p}q=p\bar{q}$ by Lemma \ref{pqexist}, 
we get
$$
  \bar{p}\bar{q}(r_q-r_p) = \bar{q}s_p - \bar{p}s_q.
$$
Unless $r_q=r_p$ the degree of the left hand side
is at least $\Deg(\bar{p})+\Deg(\bar{q})$, while the
degree of the right hand side is at most
$$
  \max\left(\Deg(\bar{q})+\Deg(s_p),\Deg(\bar{p})+\Deg(s_q)\right)
<\Deg(\bar{q})+\Deg(\bar{p}).
$$
Hence, $r_q=r_p$ and $\bar{q}s_p = \bar{p}s_q$.
Since $\bar{q},\bar{p}\not\equiv 0$ it follows that either
$s_p$ and $s_q$ are both zero or both non-zero.
Suppose  $s_p\not\equiv 0$ and $s_q\not\equiv 0$.
Write $s_p(z)=z^{m_p}\tilde{s}_p(z)$ and
$s_q(z)=z^{m_q}\tilde{s}_q(z)$ where $\tilde{s}_p(0)\neq 0$
and $\tilde{s}_q(0)\neq 0$. Since
$$
 z^{m_p}\tilde{s}_p(z)\bar{q}(z) = 
z^{m_q}\tilde{s}_q(z)\bar{p}(z)
$$
and also $\bar{q}(0)=\bar{p}(0)=1$, the lowest degree
term in the left and right hand side polynomials are $z^{m_p}$
and $z^{m_q}$ respectively, and therefore $m_p=m_q$.
Consequently,
$$
\tilde{s}_p(z)\bar{q}(z) = 
\tilde{s}_q(z)\bar{p}(z),
$$
and $\tilde{s}_p(0)=\tilde{s}_q(0)$. We can then
take $\tilde{p}(z)=\tilde{s}_p(z)/\tilde{s}_p(0)$
and $\tilde{q}(z)=\tilde{s}_q(z)/\tilde{s}_q(0)$. They
satisfy 
$$
\tilde{p}(z)\bar{q}(z) = 
\tilde{q}(z)\bar{p}(z),\qquad
\tilde{p}(0)=
\tilde{q}(0)=1,
$$
while $\Deg(\tilde{p})= \Deg(\tilde{s}_p)\leq\Deg(s_p)<\Deg(p)$
and similarly $\Deg(\tilde{q})<\Deg(q)\leq n_y$.
Hence $(\tilde{p},\tilde{q})$ is a polynomial solution
by Lemma \ref{pqexist}. It has degree strictly less than $(p,q)$, which
shows the first statement in (ii). If $\Deg(p)=\Deg(\bar{p})$
and $p\neq\bar{p}$ then there is no $r(z)$ satisfying the
requirements, showing the second statement in (ii).

\item[\rm (iii)] We finally let 
$r(z)$ be any polynomial with $\Deg(r)\leq\Deg(p)-\Deg(\bar{p})$
and $r(0)=1$. We then set
$\tilde{p}=\bar{p}r$
and $\tilde{q}=\bar{q}r$. These
polynomials trivially satisfy \eq{polsol2} and \eq{pqtaylor}.
Since $\Deg(\tilde{p})=\Deg(r)+\Deg(\bar{p})\leq \Deg(p)\leq n_x$
and 
$$
 \Deg(\tilde{q}) = 
\Deg(r)+\Deg(\bar{q}) \leq
\Deg(p)-\Deg(\bar{p})+\Deg(\bar{q})
 = \Deg(q)\leq n_y,
$$
they also satisfy \eq{polsol1} and thus are a polynomial
solution by Lemma \ref{pqexist}. In particular we can take
$r(z)$ of degree $m$.
\end{enumerate}\end{proof}

A solution to \eq{mom-olof2} can also be characterized in 
terms of the coefficient vectors. We have the following Proposition.

\begin{proposition}\label{pqcsol}
The pair $\cb=(c_0,\ldots,c_{n_x})^T\in\Real^{n_x+1}$
and $\db=(d_0,\ldots,d_{n_y})^T\in\Real^{n_y+1}$
is a coefficient solution to 
\eq{mom-olof2} if and only if 
\begin{itemize}
\item[\rm (i)] $c_0=1$, 
\item[\rm (ii)] $\cb$ is in the null-space of $A$, 
\item[\rm (iii)]
\be{cad-relation}
  d_{k} = \sum_{j=0}^{\min(k,n_x)} c_ja_{k-j}, \qquad k=0,\ldots, n_y.
\ee
\end{itemize}
\end{proposition}
\begin{proof}
Suppose first that $\cb$ is in the null-space of 
$A$, $c_0=1$ and $\{d_k\}$ is given
by \eq{cad-relation}. Extend the coefficient sequences by
setting $c_k=0$ for $k>n_x$ and
$d_k=0$ for $k>n_y$. Since $\cb$ is in the null-space of $A$, 
we get 
$\sum_{j=0}^{k} c_j a_{k-j}=0$ 
when $n_y+1\leq k\leq n_x+n_x=K$,
and in conclusion 
\be{cad-relation2} 
  d_{k} = \sum_{j=0}^{k} c_ja_{k-j}, \qquad k=0,\ldots, K.
\ee
Upon noting that $\{c_{k}\}_{k=0}^\infty$ and
$\{d_{k}\}_{k=0}^\infty$ are the Taylor coefficients of
$P_c$ and $P_d$, and since $P_c(0)=c_{0}=1$, $P_d(0)=d_{0}=a_0c_{0}=1$,
Lemma \ref{Taylor} shows that
\be{qpa}
  P_d(z) = P_c(z)\left[a_0 +a_1z +\cdots + a_Kz^K + O\left(z^{K+1}\right)\right],
\ee
and by Lemma \ref{pqexist} we have that $(P_c,P_d)$
is a solution to \eq{mom-olof2}. Conversely, if $(P_c,P_d)$
is a solution, then $c_0=P_c(0)=1$ and
by Lemma \ref{Taylor} we get that
\eq{cad-relation2} holds.
For $k=n_y+1,\ldots, K$ this also implies
that $\cb$ is in the null-space of $A$.
\end{proof}

The final Proposition of this section relates the
degree of the solution to the column vectors
of $A$ and the linear spaces they
span.

\begin{proposition}\label{avectors}
Let $V_j=\Span\{\a_1,\ldots,\a_j\}$ and 
$V^0_j=\Span\{\a_0,\ldots,\a_j\}$. Set $V_0=V^0_{-1}=\emptyset$.
Then
\begin{itemize}
\item[{\rm (i)}]
There is a solution if and only if $\a_0\in V_{n_x}={\rm Range}(A_1)$.
\item[{\rm (ii)}]
There is a solution of degree $j\geq 0$ if and only if
\be{jsol}
   \a_{0}\in V_j,  \quad {\rm and} \quad
   \a_{j}\in V^0_{j-1}.
\ee
\item[{\rm (iii)}]
When $\a_0\in V_{n_x}$ then
$$
   \a_0\in V_d, \qquad V^0_d=V_d,
$$
if and only if $d\geq\Dmax$.
\item[{\rm (iv)}]
When $\a_0\in V_{n_x}$ the vectors
$$
  \a_1,\ldots,\a_{\Dmin},
$$
(when $\Dmin>0$) 
$$
  \a_{\Dmax+1},\ldots,\a_{n_x},
$$ 
(when $\Dmax<n_x$),
are all linearly independent. Moreover,
$$
  \a_j \in V_{\Dmin}, \quad V_j=V_{\Dmin},
\qquad j=\Dmin,\ldots,\Dmax.
$$
\end{itemize}
\end{proposition}
\begin{proof}
\begin{itemize}
\item[{\rm (i)}]
By Proposition \ref{pqcsol} there exists a solution to \eq{mom-olof2}
if and only if there is a coefficient vector $\cb=(1,\; \cb')^T$
in the null-space of $A$, i.e.
$$
    A\cb = A_1\cbar + \a_0 = 0.
$$
But such a vector $\cbar$ exists if and only if $\a_0$ is
in the range of $A_1$. This shows (i).
\vskip 3 mm

\item[{\rm (ii)}]
Again by Proposition \ref{pqcsol} there is a solution of degree $j$ if
and only if there is a vector  $\cb=(c_0,c_1,\ldots,c_j,0,\ldots,0)^T$ such that
\be{ccond}
  0 = A\cb = c_0\a_0 + c_1\a_1 + \cdots + c_j\a_j,
\ee
with $c_j\neq 0$ and $c_0=1$. For $j=0$ this is clearly equivalent
to $\a_0=0$ or $\a_0\in V_0=V^0_{-1}$.
For $j>0$
the existence of $c_j$-coefficients satisfying \eq{ccond}
is equivalent to the left condition in \eq{jsol}. 
Moreover, if $\a_j\neq V^0_{j-1}=\Span\{\a_0,\ldots,\a_{j-1}\}$, 
then we must have $c_j=0$
to satisfy \eq{ccond}, and $\cb$ cannot represent a solution of degree $j$.
On the other hand, if $c_j=0$ and $\a_j=c_0'\a_0 +\cdots + c'_{j-1}\a_{j-1}$
for some non-zero coefficients $c_k'$, then $\a_0 +c_1''\a_1+\cdots +c_{j-1}''\a_{j-1}+\a_j=0$,
with $c_k''=(1+c_0')c_k-c_k'$, represents a solution of degree $j$. 
This shows (ii).

\vskip 3 mm
\item[{\rm (iii)}]The statement is obvious in case $\Dmin=0$.
If $\Dmin>0$ there are scalars such that
\be{vdef}
  \a_0 = v_1\a_1 +\cdots +v_{\Dmin}\a_{\Dmin},
\ee
by \eq{Dmin}. Hence, $\a_0\in V_{\Dmin}$ and since the
$V_j$ spaces are nested, $V_{j}\subset V_{j+1}$,
we have $\a_0\in V_d$ for $d\geq \Dmin$.
Moreover, the minimal property of $\Dmin$ ensures that $v_{\Dmin}\neq 0$
in \eq{vdef}, so that $\a_0\not\in V_d$ when $d<\Dmin$.
\vskip 3 mm

\item[{\rm (iv)}]To show that when $\Dmin>0$ the vectors
$\a_1,\ldots,\a_{\Dmin}$ are linearly independent, we use 
\eq{vdef} and note that $P_c(z)$ with
$\cb = (1,-v_1,\ldots,-v_{\Dmin},0,\ldots,0)^T$ is a polynomial
solution to \eq{mom-olof2}.
Suppose now that the there are non-zero coefficients $c'_j$ such that
$$
   c'_1\a_1 + \cdots + c'_{\Dmin}\a_{\Dmin} =0.
$$
Then $P_{c'}$ with
$\cb' = (1,c'_1-v_1,\ldots,c_{\Dmin}-v_{\Dmin},0,\ldots,0)^T$ 
is another polynomial solution to \eq{mom-olof2}.
Moreover, by the minimality property of $\Dmin$ we
must have $c_{\Dmin}-v_{\Dmin}\neq 0$ and therefore
$\Deg(P_c)=\Deg(P_{c'})=\Dmin$. But
by (ii) in Proposition \ref{pqunique} this implies that there
is yet another solution $P_{c''}$ of degree strictly
less than $\Dmin$. Hence, there are coefficients $c_j''$ such that
$$
  \a_0 + c_1''\a_1 +\cdots +c_{d}''\a_{d}=0,
$$
with $d<\Dmin$, contradicting \eq{Dmin}. The vectors must therefore be
linearly independent.

Suppose $D^*\geq \Dmin$ is the highest degree of an existing solution.
Since $P_c(z)$ is a solution of degree $\Dmin$ we get from (iii) in
Proposition \ref{pqunique} that there are solutions of all intermediate degrees
$\Dmin,\ldots, D^*$. Hence, from (ii), $\a_{j}\in V^0_{j-1}$ for
$j=\Dmin,\ldots, D^*$ and from (iii)
$\a_{j}\in V_{j-1}$ for
$j=\Dmin+1,\ldots, D^*$. Noting that
if $\a_{j+1}\in V_{j}$ then $V_j=V_{j+1}$ we can
conclude inductively that
$V_{\Dmin}=\cdots=V_{D^*}$ and $\a_j\in V_{\Dmin}$ for $j=\Dmin,\ldots,D^*$.
We now have three different cases:
\begin{enumerate}
\item If $D^*=n_x$ then $V_{\Dmin}=V_{n_x}$ and by \eq{Dmax} we get
$D^*=\Rank A_1-\Dmin+\Dmax 
= {\rm dim}\;V_{n_x}-\Dmin+\Dmax 
= {\rm dim}\;V_{\Dmin}-\Dmin+\Dmax = \Dmax$
since either $\a_1,\ldots,\a_{\Dmin}$ are linearly independent or
 $\Dmin=0$ and $V_{\Dmin}=\emptyset$. This 
shows (iv) for $D^*=n_x$.
\item If $D^*<n_x$ and $\Dmin=0$ then $V_{\Dmin}=V_{D^*}=\emptyset$
and
\be{Vnx1}
    V_{n_x} = \Span\{\a_{D^*+1},\ldots, \a_{n_x}\}.
\ee
Suppose
there are non-zero coefficients $\alpha_k$ such that
$$
  \alpha_{D^*+1}\a_{D^*+1}+\cdots +
  \alpha_{n_x}\a_{n_x} =0,
$$
and let $k^*$ be the highest index of all non-zero coefficients,
$\alpha_{k^*}\neq 0$. Then $\a_{k^*}\in V^0_{k^*-1}$ and
there is a solution of degree $k^*$ by (ii), a contradiction
to the definition of $D^*$. Hence, the vectors in \eq{Vnx1} must
be linearly independent and
$$
 D^* = n_x - {\rm dim}\;V_{n_x} = \Dmin +n_x-
\Rank A_1 =  \Dmax,
$$
showing (iv) for this case.

\item If $D^*<n_x$ and $\Dmin>0$ we have 
\be{Vnx}
    V_{n_x} = \Span\{\a_1,\ldots,\a_{\Dmin},\a_{D^*+1},\ldots, \a_{n_x}\}.
\ee
Suppose there are non-zero coefficients $\alpha_k$ such that
$$
  \alpha_1\a_1 + \cdots +
  \alpha_{\Dmin}\a_{\Dmin} + \cdots +
  \alpha_{D^*+1}\a_{D^*+1}+\cdots +
  \alpha_{n_x}\a_{n_x} =0.
$$
Since $\a_1,\ldots,\a_{\Dmin}$ are linearly independent
at least one $\alpha_k$ with $k> D^*$ must be non-zero.
By the same argument as above in case two
we then get a contradiction and 
the vectors in \eq{Vnx} must be linearly independent. Hence,
$$
 D^* = \Dmin + n_x - {\rm dim}\;V_{n_x} = \Dmin +n_x-
\Rank A_1 =  \Dmax,
$$
showing this final case.
\end{enumerate}
\end{itemize}
\end{proof}


\subsection{Proof of Theorem \ref{existence}}
\label{T2proof}

To prove Theorem \ref{existence} we essentially
have to combine the results from Propositions \ref{pqunique}
and \ref{avectors}. The statement (i) is given directly
by (i) in the latter. For the remaining points we have:
\vskip 3 mm

\begin{itemize}
\item[{\rm (ii)}] From (ii) in Proposition \ref{avectors} we see
that $\a_0\in V_d$ and $\a_d\in V^0_{d-1}$. 
It follows from (iii) in Proposition \ref{avectors} that $d\geq \Dmin$.
On the other hand,
if $\Dmax<n_x$ and $d>\Dmax$ it says that $V_{d-1}^0=V_{d-1}$.
Hence, $\a_d\in V_{d-1}$ which contradicts the linear independence
of $\a_{\Dmax},\ldots,\a_{n_x}$ established in point
(iv) of Proposition \ref{avectors}. 

\vskip 3 mm
\item[{\rm (iii)}] We note that by \eq{Dmin} 
there are scalars $v_1,\ldots,v_{\Dmin}$ such
\be{vdef2}
  \a_0 = v_1\a_1 +\cdots +v_{\Dmin}\a_{\Dmin}.
\ee
Hence, $\a_0\in V_{\Dmin}$ and since $v_{\Dmin}\neq 0$,
we also have $\a_{\Dmin}\in V^0_{\Dmin}$. By (ii) in Proposition \ref{avectors}
there is thus a solution of degree $\Dmin$ which we denote
$(p^*,q^*)$. Since
$\a_1,\ldots,\a_{\Dmin}$ are linearly independent by
(iii) in Proposition \ref{avectors},
the coefficients in \eq{vdef2} are unique and therefore
also the $\Dmin$-degree solution is unique.
Moreover, suppose that $x_j=y_i=x^*\neq 0$ for some $i,j$. Then
$p^*$ and $q^*$ would have a common factor $(1-zx^*)$,
and by Lemma \ref{pqexist} also $\bar{p}(z):={p^*}(z)/(1-zx^*)$
and $\bar{q}(z):={q^*}(z)/(1-zx^*)$ would be a solution.
But this is impossible since $\Deg(\bar{p})<\Deg(p^*)=\Dmin$.
By (iv), shown below, a solution is given by $(p^*r,q^*r)$ where
$r(0)=1$ and $\Deg(r)=\Dmax-\Dmin$. Hence 
$n_y\geq\Deg(q^*r)=\Deg(q^*) + n_x-\Rank A_1$. 
Suppose finally that $\Dmin<\Rank A_1$ and that 
$\Deg(q^*)< n_y-n_x+\Rank A_1$. Let $\Deg(r)=\Dmax+1-\Dmin$.
Then $(p^*r,q^*r)$ is still a solution by Lemma \ref{pqexist} since $(p^*,q^*)$
is a solution, $\Deg(p^*r)=\Dmax+1=n_x+\Dmin+1-\Rank A_1 \leq n_x$ and
$$
  \Deg(q^*r) < n_y-n_x+\Rank A_1 + \Dmax+1-\Dmin = n_y+1.
$$
This contradicts (ii) and therefore 
$\Deg(q^*)= n_y-n_x+\Rank A_1$, concluding the
proof of (iii).

\vskip 3 mm
\item[{\rm (iv)}] We first note that there exists a solution
of degree $\Dmax$ by Proposition \ref{avectors}
since if $\Dmax>\Dmin$ we have
$\a_0\in V_{\Dmax-1}^0$ and 
$\a_{\Dmax}\in V_{\Dmin}=V_{\Dmax-1}=V^0_{\Dmax-1}$.
Hence, (iii) in Proposition \ref{pqunique} shows that any
polynomial pair of the stated type is a solution. On
the other hand, if the polynomial solution is not
of this type, then (ii) in Proposition \ref{pqunique}
says there is a solution of degree strictly less
than $\Dmin$, contradicting (ii) above.

\vskip 3 mm

\item[{\rm (v)}] We suppose first that $A_1$ is non-singular.
Then rank\;$A_1=n_x$ so that $\Dmin=\Dmax$ and the uniqueness
is given by (iii) above. If, on the contrary, $A_1$ is singular then
$\Dmax>\Dmin$ and since we can then pick infinitely many
polynomials $r(z)$ in (iv), we have infinitely many solutions.

\vskip 3 mm

\item[{\rm (vi)}] This is a consequence of (iv). The solution can
be represented by $(p^*r,q^*r)$ for some polynomial
$r(z)$ with $r(0)=1$.
Let $1/x_j$ for $j=1,\ldots,\Dmin$
and $1/y_j$ for $j=1,\ldots,\Deg(q^*)$ be the roots of $p^*(z)$
and $q^*(z)$ respectively. Let $1/z_j$ for $j=1,\ldots,\Deg(r)$
be the roots of $r(z)$. Then 
$$
   m_k = \sum_{j=1}^{\Dmin}x_j^k + \sum_{j=1}^{\Deg(r)}z_j^k 
-\sum_{j=1}^{\Deg(q^*)}y_j^k - \sum_{j=1}^{\Deg(r)}z_j^k =
\sum_{j=1}^{\Dmin}x_j^k -\sum_{j=1}^{\Deg(q^*)}y_j^k,
$$
which is independent of $r(z)$ and uniquely determined
because $(p^*,q^*)$ is unique.\end{itemize}

\section{Proof of Theorem \ref{solvalgo}}
\label{sec:T1proof}

We can now use the results in Section \ref{sec:exist} to
prove Theorem \ref{solvalgo}.

\begin{itemize}
\item[{\rm (i-ii)}]To show the statements about Algorithms \ref{algo1} and 
\ref{algo2} we consider the reduced problem
\be{mom-red}
  m_k = \sum_{j=1}^{\tilde{n}_x} \tilde{x}_{j}^k- \sum_{j=1}^{\tilde{n}_y} \tilde{y}_{j}^k,
\qquad k=1,\ldots, \tilde{K},
\ee
where $\tilde{n}_x=\Rank A_1\leq n_x$, $\tilde{n}_y=n_y-n_x+\tilde{n}_x\leq n_y$
and $\tilde{K}=\tilde{n}_x+\tilde{n}_y\leq K$. The moments
$m_k$ in the left hand side are the same as in \eq{mom-olof2}. 
First, we consider the minimal solution $(p^*,q^*)$ of \eq{mom-olof2}.
By (iv) in Proposition \ref{avectors}
we must have $\Deg(p^*)=\Dmin\leq\Rank A_1=\tilde{n}_x$.
Moreover, by (iii) in Theorem \ref{existence},
$$
  \Deg(q^*)\leq n_y-n_x+\Rank A_1 = \tilde{n}_y.
$$
It follows from Lemma \ref{pqexist} that $(p^*,q^*)$  
is also a solution to \eq{mom-red}. Second, let
$(\tilde{p}^*,\tilde{q}^*)$ be the minimal degree
solution to \eq{mom-red}. Then by (iv) in Theorem \ref{existence}
there is a polynomial $r(z)$ with $r(0)=1$
such that $p^*=\tilde{p}^*r$
and $q^*=\tilde{q}^*r$. But then $(\tilde{p}^*,\tilde{q}^*)$
is also a solution to \eq{mom-olof2} by Lemma \ref{pqexist}.
By the uniqueness of the minimal degree solution of \eq{mom-olof2}
it follows that $r\equiv 1$ and $p^*=\tilde{p}^*$,$q^*=\tilde{q}^*$.
Suppose now that there is another polynomial $r(z)$
with $r(0)=1$, $\Deg(r)>0$ 
such that $(p^*r,q^*r)$ is a solution to 
\eq{mom-red}. Then $\Deg(p^*r)= \Dmin + \Deg(r)\leq \tilde{n}_x=
\Rank A_1$. Hence, $\Dmin<\Rank A_1$ and therefore
by (iii) in Theorem \ref{existence} we have
$\Deg(q^*)=n_y-n_x+\Rank A_1=\tilde{n}_y$.
Thus, $\Deg(q^*r)> \tilde{n}_y$ which is impossible
if $(p^*r,q^*r)$ is a solution. Hence, 
$(p^*,q^*)$ is the unique solution to \eq{mom-red} and
therefore $\tilde{A}_1$ is non-singular by (v) in
Theorem \ref{existence}. 

Since $\tilde{A}_1$ is invertible, the
generalized eigenvalue problem \eq{geneig} and $\cb'$
are well-defined. Moreover, we can construct 
$\tilde{A}_1^{-1}\tilde{A}_0$. By \eq{cprimdef},
$$
  \tilde{A}_1^{-1}\tilde{A}_0 = 
\Vector{-c_1 & 1 & 0 & \cdots &0\\
        -c_2  & 0 & 1  & \cdots &0\\
         \vdots & \vdots & \ddots & \ddots &\vdots \\
        -c_{\tilde{n}_x-1}  & 0 & 0   & \ddots &1\\
        -c_{\tilde{n}_x} & 0 & 0  & \cdots & 0
},
$$
which is a companion matrix. It is well-known that for those matrices 
the elements in the first column are the coefficients of its
characteristic polynomial. This is shown as follows: let $M_{ij}$ be the minor of 
$V:=zI-\tilde{A}_1^{-1}\tilde{A}_0$, i.e. the determinant of the
matrix obtained by removing row $i$ and column $j$.
Then, the determinant can be expanded by minors, for any $j$,
$$
   \det(V) = \sum_{i=1}^{\tilde{n}_x} (-1)^{i+j} v_{ij} M_{ij},
\qquad V=\{v_{ij}\}.
$$
Taking $j=1$, we get 
$M_{i,1}=\det(\diag(z,\ldots,z,-1,\ldots,-1))$ with $i-1$ occurrences
of $-1$, so that $M_{i,1}= z^{\tilde{n}_x-i}(-1)^{i-1}$. Therefore,
$$\begin{array}{rcl}
  \det(V) & =& (-1)^2(c_1+z)M_{1,1} + 
\sum_{i=2}^{\tilde{n}_x}
 (-1)^{i+1} c_i M_{i,1} \\ 
& = & c_1 z^{\tilde{n}_x-1}+z^{\tilde{n}_x}+\sum_{i=2}^{\tilde{n}_x}c_i z^{\tilde{n}_x-i}\\
& = & P(z),
\end{array}$$
which is exactly \eq{charpol}.
This shows that the results of Algorithms \ref{algo1}
and \ref{algo2} are identical, since the generalized
eigenvalues in \eq{geneig} are exactly the roots of $P(z)$.

It remains to show what the roots are. 
Let $\tilde{A}=[\tilde{a}_0\; \tilde{A}_1]$ be the
$A$-matrix related to \eq{mom-red}. Clearly, $\cb=(1,\cb'^T)^T$
is in the null-space of $\tilde{A}$ and hence $P_c(z)$ is
the unique solution to \eq{mom-red}. But for $z\neq 0$,
\begin{align*}
  P(z) &=  c_{\tilde{n}_x} + c_{\tilde{n}_x-1}z + \cdots + 
c_{1}z^{\tilde{n}_x-1} + z^{\tilde{n}_x} \\
 & =  z^{\tilde{n}_x}\left(\frac{c_{\tilde{n}_x}}{z^{\tilde{n}_x}}
+ \frac{c_{\tilde{n}_x-1}}{z^{\tilde{n}_x-1}} +\cdots + 
\frac{c_{1}}{z}+
1\right)\\
&= z^{\tilde{n}_x}P_c(1/z)\\
& =z^{\tilde{n}_x}(1-x_1/z)(1-x_2/z)\cdots (1-x_{\Dmin}/z) \\
&=z^{\tilde{n}_x-\Dmin}(z-x_1)(z-x_2)\cdots (z-x_{\Dmin}),
\end{align*}
which extends to $z=0$ by continuity. This concludes
the proof of points (i) and (ii).

\item[{\rm (iii)}] 
Let $(p,q)$ be a polynomial solution to \eq{mom-olof2}
and $\cb$ the corresponding coefficient solution.
From Lemma \ref{aexponent} we have
$$
   q(z) = p(z)e^{m(z)},
$$
where $m(z)$ is defined in \eq{mdef}. For the $(K+1)$-th Taylor
coefficient of the left and right hand side we have by
Lemma \ref{aexponent} and Lemma \ref{Taylor},
\be{aKrel}
  0  = \sum_{j=0}^{n_x}c_ja_{K+1-j}
\quad\Rightarrow\quad 
  a_{K+1} =  -\sum_{j=1}^{n_x} a_{K+1-j}c_j,
\ee
since the $k$-th Taylor coefficient of $q$ and $p$ is zero
for $k>n_x$ and $k>n_y$ respectively.
Finally, the last row of \eq{adef} extended to size $K+1$ gives
$$
   m_{K+1} = (K+1)a_{K+1} - \sum_{j=1}^Km_{j}a_{K+1-j}.
$$
Together the last two equations show point (iii).
\end{itemize}

\section{Properties of $A_1$ and Markov's Theorem}
\label{sec:Markov}

We now look more in detail on the structure of the $A_1$ matrix.
In particular we look at the implications of $A_1R$ being
positive definite. Then we get an explicit simplified formula for the matrix
and our results also shed some light on the relationship of
our results to the classical Markov theorem on 
the existence and uniqueness of solutions to
the finite
moment problem \eq{integrel} discussed in the introduction.
For this we need to define the matrix
$$
R = \Vector{  & & 1\\
             & \ldots & \\
            1 &&  },
$$
and note that 
left (right) multiplication by $R$ reverses the order of rows (columns) of a matrix.
In our notation we can then formulate Markov's theorem as follows
\begin{theorem}[Markov]\label{markov}
Suppose $K=2n$ is even and $n=n_x=n_y$. There is a 
unique piecewise continuous
function $f(x)$ satisfying
\be{integrel2}
   m_k = k\int_{\Real} x^{k-1}f(x)dx, \qquad 0\leq f\leq 1,
\qquad k=1,\ldots, K,
\ee
if $A_1R$ is symmetric positive definite and the matrix
\be{Acond}
 \Vector{\a_0 &  A_1\\ a_{K+1} & \a_0^T}
\ee
is singular. This $f$ is of the form in \eqtwo{fMarkov}{interlace}.
\end{theorem}
\begin{remark}
The theorem does not rule out other forms of $f(x)$ a priori,
and without the second  condition in \eq{Acond} such solutions are indeed possible.
It only considers the case $n_x=n_y$, i.e. problem \eq{mom-olof}, and says nothing
about the possibility of other solution types, e.g. when the $\{x_j\}$
and $\{y_j\}$ are not interlaced as in \eq{interlace}.
\end{remark}

We start by introducing some new notation that will be used
throughout this section. If $\{x_j\}$ and $\{y_j\}$ is a
 solution of \eq{mom-olof2} and
$(p,q)$ is the corresponding polynomial solution
as defined in \eqtwo{Pc}{Pd}, we can introduce the
new polynomials
$p_r(z) = z^{n_x}p(1/z)$ and $q_r(z)=z^{n_y}q(1/z)$
to describe the solution. Defining them by continuity at $z=0$,
we have
\be{prqr}
   p_r(z) = (z-x_1)\cdots(z-x_{n_x}),\qquad q_r(z) = (z-y_1)\cdots(z-y_{n_y}).
\ee
Furthermore, we assume that the number of {\em distinct} roots 
of $p_r$ ($x_j$-branch values) is $\tilde{n}$.
We also order the roots such that we can write
$$
   p_r(z) = (z-x_1)^{1+\eta_1}(z-x_2)^{1+\eta_2}\cdots(z-x_{\tilde{n}})^{1+\eta_{\tilde{n}}},
$$
where $1+\eta_j$ is the multiplicity of the root $x_j$, so that
$$
   n_x = \Deg(p_r) = \tilde{n} + \sum_{\ell=1}^{\tilde{n}}\eta_\ell.
$$

We start the analyis with a Lemma giving explicit expressions for the
$a_k$ values.
\begin{lemma}\label{Rationals}
For $k\geq 0$,
\be{ares1}
a_{n_y-n_x+1+k} = \sum_{j=1}^{\tilde{n}}
\frac1{\eta!}\lim_{z\to x_j}\frac{d^{\eta_j}}{dz^{\eta_j}}\frac{(z-x_j)^{1+\eta_j}z^{k}q_r(z)}{p_r(z)}.
\ee
\end{lemma}
\begin{proof}
This result follows from an application of the residue
theorem in complex analysis as follows. 
Let $C_r$ be the circle in the complex plane with
radius $r$.
Since the roots of $p(z)$ are non-zero, the function
$q/p$ is analytic within and on $C_\varepsilon$
if $\varepsilon$ is taken small enough, and
the Cauchy integral formula gives
$$
a_k =
\begin{cases}
\frac{1}{k!}\frac{d^k}{dz^k}\frac{q(z)}{p(z)}\Bigr|_{z=0},
& k\geq 0, \\
0, & k<0,
\end{cases}
=  
\frac1{2\pi i}\oint_{C_{\varepsilon}}\frac{q(z)}{p(z) z^{k+1}}dz. 
$$
Setting
\be{fdef}
  f(z) := \frac{q_r(z)}{p_r(z)} = \frac{z^{n_y-n_x}q(1/z)}{p(1/z)}.
\ee
and changing variable $z\to 1/z$ we get
$$
a_{n_y-n_x+1+k}
=\frac1{2\pi i}\oint_{C_{\varepsilon}}\frac{q(z)}{p(z) z^{n_y-n_x+k+2}}dz
= 
\frac1{2\pi i}\oint_{C_{\varepsilon}}\frac{f(1/z)}{z^{k+2
}}dz
= 
\frac1{2\pi i}\oint_{C_{1/\varepsilon}}z^{k}f(z)dz.
$$
Hence, $a_{n_y-n_x+1+k}$ is given by the sum of the residues 
of $z^{k}f(z)$ (assuming we take small
enough $\varepsilon$). By \eq{fdef} and the restriction $k\geq 0$
we see that its poles are located at
the $x_j$-values and they have multiplicities $1+\eta_j$ at $x_j$.
Then \eq{ares1} follows from the residue formula for a pole of 
a function $g(z)$ at $z^*$
with multiplicity $\eta+1$,
$$
 {\rm Res}(g,z^*) = 
\frac1{\eta!}\lim_{z\to z^*}\frac{d^{\eta}}{dz^{\eta}}(z-z^*)^{1+\eta}g(z).
$$
\end{proof}

When the branch values $\{x_j\}$ are {\em distinct} the
expression for the $a_k$ elements simplifies.
They can then be expressed as
sums of the powers of $\{x_j\}$ in a way similar to
the moments $m_k$, but with weights different from one.
We can also give a more
concise description of the matrices $A_0$ and $A_1$,
which can be factorized into a product of Vandermonde
and diagonal matrices. 
More precisely, we let $V$ be the Vandermonde matrix
$$
   V =\Vector{   1 & 1& \cdots & 1 \\
                  x_1 & x_2 & \cdots & x_{n_x} \\
                    x_1^2 & x_2^2 &\cdots & x_{n_x}^2 \\
                    \vdots & \cdots & \cdots & \vdots \\
                    x_1^{n_x-1} & x_2^{n_x-1} & \cdots & x_{n_x}^{n_x-1}},
$$
and introduce the diagonal matrices,
$$
W = \Vector{ w_1 & & \\
             & \ddots & \\
            &&  w_{n_x}
},
\qquad
X = \Vector{ x_1 & & \\
             & \ddots & \\
            &&  x_{n_x}
},
$$
where $w_j$ are the weights defined as
\be{weights}
   w_j = \frac{q_r(x_j)}{p'_r(x_j)}.
\ee
(Note that $p_r$ has only simple roots when $\{x_j\}$ are
distinct, so $p'_r(x_j)\neq 0$.)
Then we can show
\begin{proposition}\label{xjdistinct}
If $\{x_j\}$ are distinct, then for $k\geq 0$,
\be{ares2}
a_{n_y-n_x+1+k} = \sum_{j=1}^{n_x}w_j x_j^{k},
\ee
and
\be{A0A1mat}
   A_1R = VWV^T,\qquad A_0R=VWXV^T.
\ee
\end{proposition}
\begin{proof}
When $\{x_j\}$ are distinct $\eta_j=0$ for all $j$ and the
expression \eq{ares1} for the $x_j$-residue simplifies,
$$
 \lim_{z\to x_j}\frac{(z-x_j)z^{k}q_r(z)}{p_r(z)}
=\frac{x_j^{k}q_r(x_j)}{p'_r(x_j)}.
$$
This shows \eq{ares2}. For \eq{A0A1mat} we set $b_k=a_{n_y-n_x+1+k}$. Then
$$
   A_{1-r}R = \Vector{b_{r} & b_{r+1} & \hdots & b_{r+n_x} \\
 b_{r+1} & b_{r+2} & \hdots &b_{r+n_x+1} \\
            \vdots & \vdots & \ddots & \vdots \\
b_{r+n_x} & b_{r+n_x+1} & \hdots & 
b_{r+2n_x}
}\in\Real^{n_x\times n_x}, \qquad r=0,1.
$$
From \eq{ares2} we then have, for $k\geq 0$,
$$
\Vector{b_{k} \\ b_{k+1}\\ \vdots\\ b_{k+n_x} }
= \sum_{j=1}^{n_x}w_j
\Vector{x_j^{k} \\ x_j^{k+1}\\ \vdots\\ x_j^{k+n_x} }
= \sum_{j=1}^{n_x}w_jx_j^k
\Vector{1 \\ x_j\\ \vdots\\ x_j^{n_x} }
= V
\Vector{w_1x_1^k \\ w_2x_2^k\\ \vdots\\ w_{n_x}x_{n_x}^k }
= VW
\Vector{x_1^k \\ x_2^k\\ \vdots\\ x_{n_x}^k }.
$$
Consequently,
$$
   A_{1-r}R = VW
\Vector{x_1^{r} & x_1^{r+1} & \hdots & x_1^{r+n_x} \\
 x_2^{r} & x_2^{r+1} & \hdots &x_2^{r+n_x} \\
            \vdots & \vdots & \ddots & \vdots \\
x_{n_x}^{r} & x_{n_x}^{r+1} & \hdots & 
x_{n_x}^{r+n_x}
}
= VWX^rV^T,
$$
which concludes the proof.
\end{proof}

We now consider the implications of a positive
definite $A_1R$. It turns out that this is a necessary and sufficient
condition to guarantee both distinct $\{x_j\}$ values
and positive weights. 
We get

\begin{theorem}\label{SPDresults}
The matrix $A_1R$ is symmetric positive definite if and only if
$\{x_j\}$ are distinct and the weights
are strictly positive, $w_j>0$ for $j=1,\ldots, n_x$.
\end{theorem}
\begin{proof}
We use the same notation as in Lemma \ref{Rationals} and set
$$
  S_{j}(z) = \frac{1}{\eta_j!}(z-x_j)^{1+\eta_j}\frac{q_r(z)}{p_r(z)}.
$$
We note that $S_j(z)$
is smooth and regular close to $z=x_j$.
Then by Lemma \ref{Rationals}, for $k\geq 0$,
$$
  a_{n_y-n_x+1+k} = \sum_{j=1}^{\tilde{n}}\lim_{z\to x_j}\frac{d^{\eta_j}}{dz^{\eta_j}}z^{k}S_j(z).
$$
Next, we let $\v=(v_1,\ldots,v_{n_x})^T$ be an arbitrary vector 
in $\Real^{n_x}$
and recall that $P_v(z)$ is the corresponding $n_x-1$ degree polynomial
$$
   P_v(z) = v_1 + v_2z + \cdots + v_{n_x}z^{n_x-1}.
$$
Then
\begin{align}\lbeq{raylexp}
  \v^T{A}_1R\v &= \sum_{j=1}^{n_x}\sum_{k=1}^{n_x}v_jv_ka_{n_y-n_x+j+k-1}
=
\sum_{j=1}^{n_x}\sum_{k=1}^{n_x}\sum_{\ell=1}^{\tilde{n}}\lim_{z\to x_\ell}\frac{d^{\eta_\ell}}{dz^{\eta_\ell}}z^{j+k-2}S_{\ell}(z)v_jv_k
\nonumber
\\
&=
\sum_{\ell=1}^{\tilde{n}}\lim_{z\to x_\ell}\frac{d^{\eta_\ell}}{dz^{\eta_\ell}}S_{\ell}(z)\sum_{j=1}^{n_x}\sum_{k=1}^{n_x}z^{j+k-2}v_jv_k
=\sum_{\ell=1}^{\tilde{n}}\lim_{z\to x_\ell}\frac{d^{\eta_\ell}}{dz^{\eta_\ell}}S_{\ell}(z)P_v(z)^2.
\end{align}
If
\be{sizecond}
\tilde{n}+
\sum_{j=1}^{\tilde{n}}\lfloor \eta_j/2\rfloor \leq n_x-1,
\ee
we can take
$$
   P_v(z) = (z-x_1)^{1+\tilde{\eta}_1}(z-x_2)^{1+\tilde{\eta}_2}\cdots
(z-x_{\tilde{n}})^{1+\tilde{\eta}_{\tilde{n}}},
\qquad
\tilde{\eta}_j = \lfloor \eta_j/2\rfloor.
$$
Since $2(1+\tilde{\eta}_\ell) = 2+2\lfloor \eta_\ell/2\rfloor\geq 2+2( \eta_\ell/2-1)> \eta_\ell$ and
$$
   \left.\left(\frac{d^\ell}{dz^\ell}f(z)(z-z^*)^k\right)\right|_{z=z^*} = 0, \qquad 0\leq \ell <k,
$$
for all smooth enough $f(z)$, we get
$\v^T{A}_1R\v=0$, which contradicts the positivity of $A_1R$. Hence,
$$
\tilde{n}+
\sum_{j=1}^{\tilde{n}}\lfloor \eta_j/2\rfloor > n_x-1
= \tilde{n} + \sum_{\ell=1}^{\tilde{n}}\eta_\ell -1.
$$
Since for any integer $n>0$ we have $\lfloor n/2\rfloor\leq n-1$
it follows that all $\eta_\ell=0$ and $\tilde{n}=n_x$. Hence, if $A_1R$ is positive
definite, then $\{x_j\}$ are distinct.

To show the theorem it is now enough to show that, when $\{x_j\}$ are distinct,
$A_1R$ is positive if and only if the weights are positive.  From \eq{raylexp}
we then have
$$
  \v^TA_1R\v 
=\sum_{\ell=1}^{n_x}S_{\ell}(x_\ell)P_v(x_\ell)^2
=\sum_{\ell=1}^{n_x}w_\ell P_v(x_\ell)^2.
$$
Clearly, when all $w_\ell>0$, this expression is positive for $\v\neq 0$,
and $A_1R$ is positive definite. To show the converse, 
 we take $P_v(z)$ to be the Lagrange basis polynomials
$L_j(z)$ of degree $n_x-1$ defined as
$$
   L_j(x_i) = \begin{cases} 1, & i=j,\\ 0, & i\neq j.\end{cases}.
$$
If $A_1R$ is positive then
$$
  0 < \v^TA_1R\v =
\sum_{\ell=1}^{n_x}w_{\ell}L_j(x_\ell)^2 = w_j.
$$
This can be done for each $j$, which concludes the proof.
\end{proof}

We can now relate our conclusions with those in
Markov's Theorem \ref{markov}. 
We consider all solutions to \eq{mom-olof2}, instead of
those given by the integral relation \eq{integrel2}
with a piecewise continuous function $f(x)$.
The extra condition \eq{Acond} is then automatically
satisfied, and we note that the positivity of $A_1R$
guarantees a unique solution also in our space of 
density functions \eq{fWe}.
We view this as a corollary of Theorems \ref{existence} and \ref{SPDresults}.
\begin{corollary} 
If there exists a solution to \eq{mom-olof2}, then the matrix in \eq{Acond}
is singular.
When $n_x=n_y$ there is a  
unique solution to \eq{mom-olof2} of the form \eq{interlace}
if and only if $A_1R$ is symmetric positive definite.
\end{corollary}
\begin{proof}
We start by proving the singularity of \eq{Acond}.
By (ii) in Proposition \ref{pqcsol} a coefficient solution 
$\cb=(c_0,\ldots,c_{n_x})^T=(c_0,\cbar^T)^T$ 
satisfies $A\cb=0$. Since $A=(\a_0\ A_1)$
it remains to prove that $c_0a_{K+1}+\a_0^T\cbar=0$.
This was already proved in \eq{aKrel}.

Next, we prove the ``if'' part of the second statement.
If $A_1R$ is symmetric positive definite it is non-singular and
by (i), (iii) and (v) in Theorem \ref{existence} the minimal
degree solution exists and is unique and $x_j\neq y_i$ for all $i,j$. 
(If $x_j=0$ for some $j$, then there is no zero $y_i$-value
since $\Deg(q^*)=n$ by point (iii).)
By Theorem \ref{SPDresults}
the corresponding branch values $\{x_j\}$ are distinct. 
It remains to show that, upon some reordering, the
$\{x_j\}$ and $\{y_j\}$ are interlaced as in \eq{interlace}.

Order the $x_j$-values in an increasing sequence and let $m_k$ be
the number of $y_j$-values such that $y_j<x_k$. 
Clearly, $m_k$ is increasing and $0\leq m_k\leq n_y$.
Moreover, $\sign(q_r(x_k))=(-1)^{n_y-m_k}$ and
since $\lim_{z\to\infty}p'_r(z)>0$, we also have $\sign(p'_r(x_k))=(-1)^{n_x-k}$.
Hence, by also using the fact that $n_y=n_x$,
$$
  \sign(w_k) = (-1)^{n_y-m_k+n_x-k}=(-1)^{m_k+k}.
$$
We conclude that $m_k+k$ is even, which implies that $m_k$ is in fact
{\em strictly} increasing. Then, for $k=1,\ldots,n_x-1$, we have
$m_{k+1}\geq m_k+1$ and
$$
   n_x \geq m_{n_x} \geq m_k + n_x-k \quad \Rightarrow \quad m_k\leq k.
$$
Similarly, $m_k\geq m_1+ k-1\geq k-1$, so $k-1\leq m_k \leq k$, and therefore
$$
   2k-1\leq m_k+k \leq 2k.
$$
Finally, since $m_k+k$ is even we must have $m_k=k$,
which implies that the values are interlaced.

We now consider the ``only if'' part.
If there is a solution of the form \eq{interlace}, then
the $\{x_j\}$-values are obviously distinct 
and $m_k=k$. By Proposition \ref{xjdistinct}
the weights are then given by \eq{weights} and they are positive since,
as above, $\sign(w_k)=(-1)^{m_k+k}=1$. It follows from Theorem \ref{SPDresults} 
that $A_1R$ is positive definite.
\end{proof}


\section{Outlook}

Several interesting issues may be worth mentioning: 

\begin{enumerate}
\item Computational complexity in a finite difference implementation: one can consult the article \cite{GO} where practical implementation issues and several examples of increasing complexity have been addressed in the context of geometric optics problems. In particular, comparisons with Lagrangian (ray-tracing) solutions are shown.

\item Extension to higher dimensions: nothing seems to be done in this direction at the time being; see however the last sections of \cite{put} and the routines based on complex variables in \cite{GM,milan} for ``shape from moments".

\item A very special case of the trigonometric moment problem can be solved by means of a slight variation of the algorithms presented here, in \cite{GO} and in Section IV.A of \cite{milan}. That is to say, one tries to invert the following set of equations:
\begin{equation}
\label{trigo}
\sum_{j=0}^n \mu_j \exp(ik\lambda_j)=m_k, \qquad k=0, ..., n.
\end{equation}
Let us state that in case the $n+1$ real frequencies $\lambda_j$ are known, the set of complex amplitudes $\mu_j$ are found by solving a Vandermonde system:
$$
\left(\begin{array}{ccc}
1 & \cdots & 1 \\
\exp(i\lambda_0) & \cdots & \exp(i\lambda_n) \\
\vdots & & \vdots \\
\exp(in\lambda_0) & \cdots & \exp(in\lambda_n) \\
\end{array}\right)
\left(\begin{array}{c}\mu_0 \\ \mu_1 \\ \vdots \\ \mu_n \\ 
\end{array}\right)=\left(\begin{array}{c}
m_0 \\ m_1 \\ \vdots \\ m_n \\ \end{array}\right).
$$
The frequencies can be found through a byproduct of \cite{milan,GO} as we state now: let us suppose $n$ is odd ({\it i.e.} the number of equations is even), we form the two matrices,
$$
A_1=\left(\begin{array}{ccc}
m_0 & \cdots & m_{\frac{n-1}2} \\
\vdots & & \vdots \\
m_{\frac{n-1}2} & \cdots & m_{{n-1}} \\
\end{array}\right), \qquad 
A_2=\left(\begin{array}{ccc}
m_1 & \cdots & m_{\frac{n+1}2} \\
\vdots & & \vdots \\
m_{\frac{n+1}2} & \cdots & m_{{n}} \\
\end{array}\right),
$$
and then the frequencies can be obtained through a generalized eigenvalue problem, $A_1\v_j=\lambda_j A_2 \v_j$, $j=0, ..., n$. This kind of algorithm can be used to check the accuracy of the classical FFT and will be studied in a forthcoming article.

\end{enumerate}


\end{document}